\documentclass[11pt]{amsart}

\usepackage{epigamath}

\usepackage[english]{babel}

\numberwithin{equation}{section}

\usepackage{enumitem}
\usepackage{amsmath, amsfonts, amssymb}
\usepackage{mathtools}
\usepackage[all]{xy}
\usepackage[utf8]{inputenc}
\usepackage{varioref}
\usepackage{xparse}
\usepackage{graphicx}
 
\theoremstyle{plain}
\newtheorem{theorem}{Theorem}[section]
\newtheorem{lemma}[theorem]{Lemma}

\newtheorem{proposition}[theorem]{Proposition}

\theoremstyle{remark}
\newtheorem{question}[theorem]{Question}
\newtheorem{example}[theorem]{Example}
\newtheorem{remark}[theorem]{Remark}

\theoremstyle{definition}
\newtheorem{conjecture}[theorem]{Conjecture}
\newtheorem{definition}[theorem]{Definition}

\newcommand{\diam}{{{\operatorname{diam}}}}
\newcommand{\NEbar}{{\overline{\operatorname{NE}}}}
\newcommand{\Amp}{{{\operatorname{Amp}}}}
\newcommand{\Movbar}{{\overline{\operatorname{Mov}}}}

\DeclarePairedDelimiter\floor{\lfloor}{\rfloor}

\newcommand{\sK}{{\mathcal K}}
\newcommand{\sO}{{\mathcal O}}

\newcommand{\scrB}{{\mathscr B}}
\newcommand{\scrU}{{\mathscr U}}
\newcommand{\scrX}{{\mathscr X}}
\newcommand{\scrY}{{\mathscr Y}}

\newcommand{\C}{{\mathbb C}}
\renewcommand{\P}{{\mathbb P}}
\newcommand{\Q}{{\mathbb Q}}
\newcommand{\R}{{\mathbb R}}
\newcommand{\Z}{{\mathbb Z}}

\newcommand{\gothF}{{\mathfrak F}}
\newcommand{\gothG}{{\mathfrak G}}

\newcommand{\abs}[1]{{\left|#1\right|}}
\newcommand{\Aut}{\operatorname{Aut}}
\newcommand{\chern}{{\rm c}}
\newcommand{\codim}{\operatorname{codim}}
\newcommand{\Def}{{\operatorname{Def}}}
\newcommand{\eps}{\varepsilon}
\newcommand{\isom}{{\ \cong\ }}
\newcommand{\lt}{{\rm{lt}}}
\newcommand{\Nef}{{\operatorname{Nef}}}
\newcommand{\NE}{{\operatorname{NE}}}
\renewcommand{\O}{{\rm O}}
\newcommand{\ohne}{{\ \setminus \ }}
\newcommand{\ol}[1]{{\overline{#1}}}
\newcommand{\Pic}{\operatorname{Pic}}
\newcommand{\ratl}{\dashrightarrow}
\newcommand{\reg}{{\operatorname{reg}}}
\newcommand{\rk}{\operatorname{rk}}
\newcommand{\sing}{{\operatorname{sing}}}
\newcommand{\Sp}{{\operatorname{Sp}}}
\newcommand{\tensor}{\otimes}
\newcommand{\veps}{\varepsilon}
\newcommand{\vrho}{{\varrho}}
\newcommand{\wt}[1]{{\widetilde{#1}}}

\newcommand{\rrk}{{\operatorname{rrk}}}
\newcommand{\Mov}{{\operatorname{Mov}}}
\def\Mon{{\operatorname{Mon}}}

\newcommand{\lra}{\longrightarrow}
\setlist[enumerate,1]{label={\rm(\arabic*)}, ref={\rm\arabic*}}

\usepackage{url}

\EpigaVolumeYear{9}{2025} \EpigaArticleNr{22} \ReceivedOn{March 2, 2023}
\InFinalFormOn{March 12, 2024}
\AcceptedOn{September 19, 2024}

\title{Non-hyperbolicity of holomorphic symplectic varieties}
\titlemark{Non-hyperbolicity of holomorphic symplectic varieties}

\author{Ljudmila Kamenova}
\address{Department of Mathematics, 3-115, 
Stony Brook University, Stony Brook, NY 11794-3651, USA}
\email{kamenova@math.stonybrook.edu}

\author{Christian Lehn}
\address{Fakult\"at f\"ur Mathematik, Ruhr-Universit\"at Bochum, Universit\"ats\-stra\ss e~150, Postfach IB 45, 44801 Bochum, Germany}
\email{christian.lehn@rub.de}

\authormark{L.~Kamenova and C.~Lehn}

\AbstractInEnglish{We prove the non-hyperbolicity of primitive symplectic varieties with \mbox{$b_2 \geq 5$} that satisfy the rational SYZ conjecture. If in addition $b_2 \geq 7$, we establish that the Kobayashi pseudometric vanishes identically. This in particular applies to all currently known examples of irreducible symplectic manifolds and thereby completes the results by Kamenova--Lu--Verbitsky. The key new contribution is that a projective primitive symplectic variety with a Lagrangian fibration has vanishing Kobayashi pseudometric. The proof uses ergodicity, birational contractions, and cycle spaces.}

\MSCclass{32Q45, 53C26 (primary); 14B07, 14J10, 32S45 (secondary)}
\KeyWords{Hyperk\"ahler manifold, symplectic variety, hyperbolicity, Kobayashi pseudometric, locally trivial deformation, birational contraction, Lagrangian fibration, ergodicity, cycle space}

\acknowledgement{Ljudmila Kamenova was partially supported by a grant from the Simons Foundation/SFARI (522730, LK) 2017--2024. 
Christian Lehn was supported by the DFG research grant Le 3093/3-2 and the SMWK research grant SAXAG.}

\begin{document}



\maketitle

\begin{prelims}

\DisplayAbstractInEnglish

\bigskip

\DisplayKeyWords

\medskip

\DisplayMSCclass

\end{prelims}


\newpage

\setcounter{tocdepth}{1}

\tableofcontents


\section{Introduction}\label{section intro}
\thispagestyle{empty}

The Kobayashi pseudometric on a complex variety is the maximal pseudometric such that any holomorphic map from the Poincar\'e disk to the variety is distance decreasing. It is a fundamental object and of great interest in complex geometry. A variety is called \emph{Kobayashi hyperbolic} if this pseudometric is a genuine metric, \textit{i.e.}, if it is non-degenerate. Kobayashi's conjectures  \cite[Section~13(F), Problem~F.2, p.~405]{Ko76} predict that for Calabi--Yau varieties, the opposite is the case: This pseudometric vanishes identically.

In this article, we study non-hyperbolicity and vanishing of the Kobayashi pseudometric of compact K\"ahler holomorphic symplectic varieties. While Verbitsky \cite{Ver15, Ver17} has shown that any irreducible symplectic manifold with second Betti number $b_2 \geq 5$ is non-hyperbolic (building on Campana's result that any twistor family of irreducible symplectic manifolds contains at least one non-hyperbolic member, see \cite[Theorem~1]{Cam92}), a stronger statement has been shown by Kamenova--Lu--Verbitsky \cite{KLV14} under some additional geometric assumptions. More precisely, they prove that irreducible symplectic manifolds with second Betti number at least $13$ satisfying the hyperk\"ahler version of the SYZ conjecture (see Conjecture~\ref{conjecture syz}) have vanishing Kobayashi pseudometric; see \cite[Theorem~1.2]{KLV13v3}. Their strategy is to deform to a variety admitting two transversal Lagrangian fibrations and then use ergodicity and the upper semi-continuity of the Kobayashi diameter, see \cite[Corollary~1.23]{KLV14}, to transport the result to any other manifold of the same deformation type. 

The purpose of this article is to improve the Kamenova--Lu--Verbitsky bound on the second Betti number in order to obtain the vanishing of the Kobayashi pseudometric for all known examples of irreducible symplectic manifolds. Our key discovery is that for the pseudometric to vanish it is already enough to have \emph{one} Lagrangian fibration instead of two; see Theorem~\ref{theorem lagrangian fibration} for a precise statement. With this at hand, we can prove our main result (see Theorem~\ref{theorem main} for a slightly stronger statement). 

\begin{theorem}\label{theorem main introduction}
Let $X$ be a primitive symplectic variety. Suppose that every primitive symplectic variety which is a locally trivial deformation of\, $X$ satisfies the rational SYZ conjecture. Then the following hold:  
\begin{enumerate}
    \item If\, $b_2(X) \geq 5$, then $X$ is non-hyperbolic.
    \item If\, $b_2(X) \geq 7$, then the Kobayashi pseudometric $d_X$ vanishes identically.
\end{enumerate}
\end{theorem}

Notice that our results are valid for singular holomorphic symplectic varieties as well; see Section~\ref{section symplectic} for the precise definitions. In fact, singular varieties are the natural context for our arguments. The proof of Theorem~\ref{theorem lagrangian fibration} for example crucially needs to pass through the singular world, even if you start with a smooth variety. For smooth varieties, the main result, Theorem~\ref{theorem main introduction}, could be proven by modifying the arguments slightly so as to (mostly) avoid singularities, but formulating and proving it for primitive symplectic varieties leads to greater clarity. 

In view of the decomposition theorem \cite[Theorem~A]{BGL22}, see also \cite[Theorem~1.5]{HP19}, it is natural to ask whether the vanishing of the Kobayashi pseudometric holds for any compact K\"ahler holomorphic symplectic variety. The following result is an easy consequence of the decomposition theorem and also justifies why we may restrict our attention to primitive (or even irreducible) symplectic varieties. 

\begin{proposition}
If the Kobayashi pseudometric vanishes for every irreducible symplectic variety, then the same holds true for any compact K\"ahler holomorphic symplectic variety.
\end{proposition}

This result is proven as Proposition~\ref{proposition decomposition theorem}. As every irreducible symplectic variety is primitive symplectic, it would in particular be sufficient to get rid of the assumptions on $b_2$ and on the validity of the SYZ conjecture in Theorem~\ref{theorem main introduction}. Removing these hypotheses would however require a new idea.

\subsection{Outline of the argument}

As in \cite{KLV14}, the idea is to first prove the vanishing of the Kobayashi pseudometric for a special class of primitive symplectic varieties and then, after having obtained this ``initial'' vanishing statement, deduce the Kobayashi conjecture for all primitive symplectic varieties of the same (locally trivial) deformation type. 

While Kamenova--Lu--Verbitsky used irreducible symplectic varieties admitting two transversal Lagrangian fibrations, we show that  a single Lagrangian fibration is already sufficient. Given two transversal fibrations, the vanishing statement is an obvious consequence of the triangle inequality for the Kobayashi pseudometric. The drawback is that assuring the existence of two fibrations increases the second Betti number (although we suspect that the approach in \cite{KLV14} can be pushed to get better bounds). Improving their result to just one fibration is the main new contribution of this work and occupies the largest part of the article. We will elaborate on this part below, but let us first explain how to conclude the proof of the main result.

Assuming the SYZ conjecture, the existence of Lagrangian fibrations reduces to a lattice-theoretic question, which by Meyer's theorem has a positive answer for a lattice of rank at least $5$. Incidentally,  the ergodicity properties of periods also require the hypothesis $b_2 \geq 5$. From there we follow the argument of Kamenova--Lu--Verbitsky with some minor modifications due to the presence of singularities. Ergodicity is then used to transport the vanishing of the Kobayashi metric from varieties admitting Lagrangian fibrations to all varieties of the same locally trivial deformation type, using the aforementioned upper semi-continuity of the Kobayashi diameter; see \cite[Corollary~1.23]{KLV14}. The semi-continuity was proven for families of smooth varieties, so at this point the existence of simultaneous resolutions in locally trivial families proven in \cite[Corollary~2.27]{BGL22} comes in handy.

Coming back to the ``initial'' vanishing statement for varieties admitting a Lagrangian fibration, let us illustrate our strategy with the following simple example.

\begin{example}\label{example elliptic k3}
  Let $f\colon S\to \P^1$ be an elliptic K3 surface with a section $E\subset S$. Then $S$ is chain connected by subvarieties with vanishing Kobayashi metric; hence $d_S\equiv 0$. We are however looking for a different interpretation of this argument as, in higher dimensions, we do not want to assume our fibrations to have sections. Instead, we divide the problem in two. First, we will contract $E$ and thus obtain a birational map $\pi\colon S \to \bar S$. Let us observe that now all (images of) fibers of $f$ meet in a single point. Hence, $\bar S$ is chain connected by an irreducible family of cycles with vanishing Kobayashi pseudometric; in particular, $d_{\bar S}\equiv 0$. As a second step, we remark that, in this situation, the problem is invariant under birational maps, and thus also conclude  $d_S\equiv 0$. This point of view generalizes to higher dimensions.
\end{example}

Even though the above example is very simple, the general strategy is rather similar to the one illustrated in the example. First, we show that given a (rational) Lagrangian fibration, either there is a second one that is distinct from it, or our variety has non-trivial divisorial contractions. In the latter case, the ultimate goal is to show that the given fibration ceases to be almost holomorphic (see Definition~\ref{definition almost holomorphic}) on some birational model. Then we use cycle spaces and Campana's theorem on almost holomorphic maps to conclude that the resulting rational Lagrangian fibration on the contracted variety is chain connected by its fibers (as is the singular K3 surface $\bar S$ in Example~\ref{example elliptic k3}). As the Kobayashi pseudometric vanishes when restricted to the fibers, we infer the vanishing of the Kobayashi pseudometric of our holomorphic symplectic variety.

\subsection{Organization of the article}

In Section~\ref{section symplectic}, we recall definitions of (singular) holomorphic symplectic varieties, the Beauville--Bogomolov--Fujiki (or BBF for short) form on the second cohomology and its properties, as well as some background on Lagrangian fibrations. None of the material is new; we however carefully compile the results that lead to the proof of Matushita's theorem for primitive symplectic varieties, see Theorem~\ref{theorem matsushita}, and we discuss the relation between the different versions of the SYZ conjecture in Section~\ref{section syz}. Section~\ref{section hyperbolicity} is of preliminary nature as well and provides basic notions concerning hyperbolicity and properties of the Kobayashi pseudometric. The purpose of Section~\ref{section almost holomorphic} is to explain Campana's theorem on almost holomorphic maps and to adapt a result from \cite{GLR13} on almost holomorphic Lagrangian fibrations to primitive symplectic varieties; see Theorem~\ref{theorem glr}.

The new contributions of this article are contained in Section~\ref{section non-hyperbolicity results}. Here, our main result, Theorem~\ref{theorem main}, is proven. As explained before, we assume the second Betti number to be at least~$7$. Unlike in the smooth case, there are examples of singular primitive symplectic varieties (even $\Q$-factorial, terminal ones) with $b_2(X)$ strictly smaller than~$7$. We illustrate some of these in Section~\ref{section examples}.

\subsection*{Conventions} 
A \emph{variety} will be a reduced complex Hausdorff space which is countable at infinity.\footnote{That is, it is a countable union of compact subspaces. This property is also known as $\sigma$-compactness.} An \emph{algebraic variety} over a field $k$ is a reduced scheme that is separated and of finite type over $k$. A resolution of singularities of a variety $X$ is a proper, bimeromorphic morphism $\pi\colon Y \to X$ such that $Y$ is a smooth variety. We denote by~$\Omega_X^p$ the sheaf of holomorphic $p$-forms on $X$ and by $\Omega_X^{[p]}$ its double dual, the sheaf of \emph{reflexive} (holomorphic) $p$-forms. A complex variety $X$ is called $\Q$-factorial if for every reflexive sheaf $L$ on $X$ of rank~$1$, there is a positive integer $n$ such that the double dual~$(L^{\tensor n})^{\vee\vee}$ is invertible.

\subsection*{Acknowledgments.} 
We would like to thank Ben Bakker for helpful conversations around the Kobayashi pseudometric, St\'ephane Druel and Daniel Greb for helpful discussions improving Section~\ref{section symplectic}, Ariyan Javanpeykar for pointing out Example~\ref{example ariyan} to us, Giovanni Mongardi for Example~\ref{Giovanni example}, Steven Lu for the reference on length functions on complex spaces, and Claire Voisin for pointing out a confusion in one of the arguments. We are grateful to the referee for their thorough reading and for drawing our attention to a few inaccuracies in the original manuscript.

\section{Holomorphic symplectic varieties}\label{section symplectic}

This section provides a brief recollection of holomorphic symplectic varieties. Let us begin by recalling the notion of an irreducible symplectic manifold.

\begin{definition}\label{definition irreducible symplectic manifold}
An \emph{irreducible symplectic manifold} is a connected compact complex K\"ahler manifold $M$ satisfying $\pi_1(M)=0$ and $H^{2,0}(M)={\mathbb C} \sigma$, where $\sigma$ is a holomorphic symplectic form.
\end{definition}

These manifolds are sometimes referred to as \emph{compact hyperk\"ahler manifolds}, which is an equivalent concept. Indeed, in every K\"ahler class on an irreducible symplectic manifold, there is a unique hyperk\"ahler metric (\textit{i.e.}, with holonomy equal to $\Sp(n)$) by Yau's theorem. Conversely, a compact hyperk\"ahler manifold is irreducible symplectic for a $\P^1$ family of complex structures.

Let us now come to \emph{singular} holomorphic symplectic varieties.

\begin{definition}\label{definition symplectic}
A \emph{primitive symplectic variety} is a normal compact K\"ahler variety $X$ with rational singularities such that $H^1(X,\sO_X)=0$ and $H^0(X,\Omega_X^{[2]})=\C\sigma$ for a holomorphic symplectic\footnote{A reflexive $2$-form is called \emph{symplectic} if its restriction to the regular part is.} form $\sigma$.

An \emph{irreducible symplectic variety} is a normal compact K\"ahler variety $X$ with rational singularities such that for each finite, quasi-\'etale\footnote{Recall that \emph{quasi-\'etale} means \'etale in codimension~$1$.} cover $\pi\colon X'\to X$, the algebra $\Gamma(X',\Omega_{X'}^{[\bullet]})$ of global reflexive holomorphic forms is generated by a symplectic form $\sigma'\in \Gamma(X',\Omega_{X'}^{[2]})$.
\end{definition}

The notion of an irreducible symplectic variety is due to Greb--Kebekus--Peternell; see \cite[Definition~8.16]{GKP16}, where we just replaced the projectivity assumption by the requirement for $X$ to be compact K\"ahler. Irreducible symplectic is more restrictive than primitive symplectic and serves a different purpose: Irreducible symplectic varieties are one of the three fundamental building blocks in the decomposition theorem (see \cite{HP19, BGL22} and references therein), whereas for primitive symplectic varieties, moduli theory essentially works as in the smooth case (see \cite{BL21, BL22}). 

\subsection{Deformations of holomorphic symplectic varieties}\label{section symplectic defo}

We briefly discuss locally trivial deformations, especially for primitive symplectic varieties. For details and further references, we refer to \cite[Section~4]{BL22}. 

\begin{definition}\label{definition locally trivial deformation}
Let $f\colon \scrX \to S$ be a \emph{deformation} of a compact complex space $X$, that is, a flat and proper morphism of complex spaces with target a connected complex space $S$ such that $X=f^{-1}(0)$ for a distinguished point $0\in S$. Such a deformation is \emph{locally trivial} if for every $x \in X$ there exist open neighborhoods $\scrU \subset \scrX$ of $x$ and $V \subset S$ of $0$ such that $\scrU \isom (\scrU \cap X) \times V$ over $S$. 
\end{definition}

Note that every deformation of a compact complex manifold is locally trivial. Moreover, locally trivial deformations $\scrX \to S$ are globally trivial in the real analytic category after shrinking the base $S$; see \cite[Proposition~5.1]{AV21}. In particular, they are topologically trivial.

Let $\scrX \to \Def(X)$ be the miniversal deformation of $X$ (in the sense of space germs). The base $\Def(X)$ of this deformation is referred to as the \emph{Kuranishi space} of $X$. By \cite[(0.3) Corollary]{FK87}, there is a maximal closed subspace $\Def^\lt(X)\subset \Def(X)$ parametrizing locally trivial deformations and the restriction of the miniversal family to this subspace is miniversal for locally trivial deformations of $X$. Further recall that the tangent space to $\Def^\lt(X)$ at the distinguished point $0$ is isomorphic to $H^1(X,T_X)$ and that every miniversal locally trivial deformation of $X$ is universal if $H^0(X,T_X)=0$. For primitive symplectic varieties, the following result summarizes some fundamental properties of locally trivial deformations.

\begin{theorem}\label{corollary symplectic locally trivial}
Let $X$ be a primitive symplectic variety. Then $X$ admits a universal locally trivial deformation $\scrX \to \Def^\lt(X)$. Moreover, $\Def^\lt(X)$ is smooth of dimension $h^{1,1}(X)$, all fibers are again primitive symplectic $($after possibly shrinking $\Def^\lt(X))$, and the universal deformation remains universal for any of its fibers.
\end{theorem}
\begin{proof}
This is Lemma~4.6, Theorem~4.7, and Corollary~4.11 of \cite{BL22}.
\end{proof}


\subsection{The Bogomolov--Beauville--Fujiki form}\label{section bbf form}

Given a primitive symplectic variety $X$, there is the \emph{Beauville--Bogomolov--Fujiki} (\emph{BBF}\,) \emph{form}
\[
q_X\colon H^2(X,\Z) \lra \Z, 
\]
which is a quadratic form that generalizes the intersection pairing for K3 surfaces. As in the smooth case, it carries a lot of information about the variety in question. We refer to \cite[Section~5]{BL22} for the explicit formula defining it, several basic results (such as the proof that it is actually an \emph{integral} quadratic form), and for references to many earlier partial results about this form. Here, we content ourselves with listing its most important properties.

\begin{lemma}\label{lemma bbf form}
  The BBF form $q_X$ on a primitive symplectic variety $X$ has the following properties:
\begin{enumerate}
\item It is invariant under locally trivial deformation.
\item It is non-degenerate of signature $(3,b_2(X)-3)$.
\item\label{lbf-3} On the real space underlying $H^{2,0}(X)\oplus H^{0,2}(X)$, the form is positive definite.
\item The orthogonal complement of\, $(H^{2,0}(X)\oplus H^{0,2}(X))$ equals $H^{1,1}(X)$.
\item\label{lbf-5} The Fujiki relation holds; i.e., there is a constant $c\in \Z$, which is invariant under locally trivial deformation, such that $$\int_X \alpha^{2n}=c\cdot q_X(\alpha)^n$$ for any $\alpha \in H^2(X, {\mathbb Z})$.
\end{enumerate}
\end{lemma}
\begin{proof}
We again refer to \cite[Section~5]{BL22} and the references therein, in particular, Lemmas~5.3 and~5.7.
\end{proof}

\noindent
Notice that the restriction of the form $q_X$ to $H^{1,1}(X, \mathbb R)$ has signature \mbox{$(1,b_2(X)-3)$} because of Lemma~\ref{lemma bbf form}\eqref{lbf-3}. Therefore, the cone $$\{ \alpha \in H^{1,1}(X, {\mathbb R}) ~|~ q_X(\alpha) > 0 \} $$ has two connected components. 

\begin{definition}
The positive cone $C_X \subset H^{1,1}(X, \mathbb R)$ of a primitive symplectic variety~$X$ is the connected component of the cone $\{ \alpha \in H^{1,1}(X, {\mathbb R}) ~|~ q_X(\alpha) > 0 \} $ containing the K\"ahler cone $\sK_X$ of $X$. 
\end{definition}

\begin{definition}\label{definition picard cones}
For a primitive symplectic variety $X$, let $\Pic (X)_\R$ be the real Picard group $\Pic (X) \otimes \mathbb R$. Inside $\Pic (X)_\R$, we consider the following cones:
\begin{itemize}
    \item The \emph{ample cone} $\Amp (X)$ of $X$ is the cone generated by all ample (integral) Cartier divisors on X. 
    \item The \emph{nef cone} $\Nef(X)$ of $X$ is the intersection of the closure of the K\"ahler cone~$\sK_X$ with the real Picard group $\Pic(X)_\R$.
    \item The \emph{movable cone} $\Mov(X)$ of $X$ is the cone generated by all movable line bundles (\textit{i.e.}, whose linear system is non-empty and has no fixed part) in $\Pic (X)_\R$. We denote its closure by $\Movbar(X)$. 
\end{itemize}
\end{definition}

\noindent Note that, in general, the movable cone is neither open nor closed. Also, our definition of the nef cone is slightly non-standard, for usually it is defined as the closure of the ample cone. If $X$ is projective, both definitions coincide. If however the ample cone is zero, there can still be non-trivial nef line bundles. As an example, take a primitive symplectic variety of Picard rank~$1$ admitting a Lagrangian fibration.

\begin{definition}
Let $N_1(X)_\R$ denote the space of $1$-cycles (with real coefficients) modulo numerical equivalence. We furthermore define the cone $\NE(X)\subset N_1(X)_\R$ to be the cone generated by the classes of effective $1$-cycles and let $\NEbar(X)$ denote its closure. The cone $\NEbar(X)$ is called the \emph{Mori cone} of $X$. 
\end{definition}

\subsection{Lagrangian fibrations}\label{section lagrangian fibrations}

A subvariety $Y$ of a holomorphic symplectic manifold $(X,\sigma)$ is \emph{Lagrangian} if $\dim Y=\frac{1}{2} \dim X$ and the restriction of $\sigma$ to the regular locus $Y^\reg$ vanishes. This is equivalent to saying that the pullback of $\sigma$ to a resolution of singularities of $Y$ vanishes. For singular $X$, one can extend this notion in an obvious way to subvarieties not contained in the singular locus $X^\sing$. However, as by definition all our symplectic varieties have rational singularities, we can do better. Thanks to \cite[Theorem~1.10]{KS21}, there is a functorial pullback for reflexive differentials. Hence, we can define Lagrangian subvarieties in full generality.

\begin{definition}\label{definition lagrangian}
Let $(X,\sigma)$ be a holomorphic symplectic variety. A subvariety $Y \subset X$ is called \emph{Lagrangian} if $\dim Y=\frac{1}{2} \dim X$ and the Kebekus--Schnell pullback of $\sigma$ to a resolution of singularities of $Y$ vanishes.
\end{definition}

The following theorem is due to Matsushita in the smooth case; see \cite{Mat99,Mat99add, Mat00, Mat03}, and Hwang~\cite{Hwa08} for the last statement. Subsequently, results for singular varieties were obtained by Matsushita \cite{Mat15} and Schwald \cite{Sch20}. We summarize their results and include a sketch of a proof, in part because some of the results hold in greater generality than originally stated.

\begin{theorem}\label{theorem matsushita}
Let $X$ be a primitive symplectic variety of dimension $2n$, and  let $f\colon  X \rightarrow B$ be a surjective holomorphic map with connected fibers to a normal K\"ahler variety $B$ with $0<\dim B < 2n$. Then the following hold: 
\begin{enumerate}
    \item\label{tm-1} The base $B$ is a projective variety with Picard rank $\vrho(B)=1$; in particular, $B$ is projective and has  $\Q$-factorial, log terminal singularities. Furthermore, $\dim B=n$.
    \item\label{tm-2} The morphism $f$ is equidimensional, and each irreducible component of each fiber of $f$ endowed with the reduced structure is a Lagrangian subvariety. The singular locus $X^\sing$ does not surject onto $B$.
    \item\label{tm-3} All smooth fibers are abelian varieties.
    \item\label{tm-4} If, in addition, $X$ is irreducible symplectic, then $B$ is Fano. Moreover, if\, $B$ is smooth, then $B\isom \P^n$.
    \end{enumerate}
\end{theorem}
\begin{proof}
The argument in \cite{Mat03} shows that $B$ is actually projective by first showing that it has log terminal, hence rational, singularities, and then that it is Moishezon. As in \cite{Mat99add}, one shows that the general fiber of $f$ is Lagrangian (and hence a complex torus of dimension $n$) so that $\dim B = n$. With the argument of \cite{Mat99}, one deduces that $B$ is $\Q$-factorial of Picard rank~$1$.

For~\eqref{tm-2}, let $\rho\colon Y \to X$ be a resolution of singularities. By \cite[Theorem~2.1]{Kol86}, and \cite[Theorem~2.3, Remark~2.9]{Sai90} in the analytic case, the derived direct images $R^i(f\circ\rho)_*\omega_Y$ are torsion-free for $i\geq 0$. As $X$ has rational singularities, we have $R\rho_*\omega_Y=\omega_X$, so  the  $R^if_*\omega_X$ are also torsion-free. From there, the proof of equidimensionality and Lagrangeness is essentially the same as \cite[Theorem~1]{Mat00}. 

To see that $X^\sing$ does not dominate the base, we adapt the proof of Matsushita's ``Theorem of Matreshka''; see \cite[Theorem~3.1]{Mat15}. The crucial point is that the singular locus of a holomorphic symplectic variety is a Poisson subvariety (for the Poisson structure induced on $X$ by the symplectic form); see \cite[Theorem~2.3]{Kal06b} and also \cite[Theorem~3.4]{BL22} for an adaption to the complex analytic setting. We consider the diagram
\[
\xymatrix{
X_1 \ar[r] \ar[d]_{f_1} & X \ar[d]^{f} \\
B_1 \ar[r] & B\rlap{,} \\
}
\]
where we denote by $X_1$ the normalization of $X^\sing$ and by $B_1$ the normalization of $f(X^\sing)$. Then $f$ being Lagrangian implies that pullbacks of functions in $\sO_B$ Poisson commute. As $X_1$ is a Poisson subvariety, the Poisson structure is compatible with restriction, so the $f_1$-pullbacks of functions in $\sO_{B_1}$ also Poisson commute. Hence, coordinate functions around a smooth point of $B_1$ give $\dim B_1$ linearly independent Hamiltonian vector fields whose action preserves the fibers of $f_1$. Therefore, the fibers of $f_1$ have dimension at least $\dim B_1$. In particular, $\dim B_1 < \dim B$, which implies the claim.

From the classical theory of integrable systems, it follows that the smooth fibers are complex tori. The projectivity statement in~\eqref{tm-3} follows as in the smooth case by Voisin's argument; see \cite[Proposition~2.1]{Cam06}. 

\eqref{tm-4}~ The proof of this part is essentially identical to the proof of \cite[Lemma~2.2]{Mat03}. For the existence of a singular K\"ahler--Einstein metric on $X$ that is smooth on the regular part $X^\reg$, we refer to \cite[Theorem~A]{EGZ09} and \cite[Corollary~1.1]{Pau08}. The last statement has been proven by Hwang \cite{Hwa08} if the total space $X$ is smooth and projective, building on work of Matsushita \cite{Mat05}. For singular projective $X$, this is due to Matsushita \cite{Mat15}; see also \cite{Sch20}. Finally, Greb and the second-named author treated the case of a smooth K\"ahler total space $X$ in \cite{GL14}. Their argument, however, works literally the same if $X$ is singular and K\"ahler.
\end{proof}

\begin{definition}\label{definition lagrangian fibration}
Let $X$ be a primitive symplectic variety. A map $f\colon X \to B$ as in Theorem~\ref{theorem matsushita} is called a ({\it holomorphic}\,) {\it Lagrangian fibration}. A \emph{rational Lagrangian fibration} is a meromorphic map $f\colon X \ratl B$ to a normal K\"ahler variety $B$ such that $f$ has connected fibers\footnote{Recall that a \emph{fiber} of a meromorphic map $f\colon X\ratl B$ is the Zariski closure of a fiber of the restriction of $f$ to its domain of definition. In particular, fibers are always compact if $X$ is.} and its general fiber is a Lagrangian subvariety of $X$. For a movable line bundle $L$, we say that a (rational) Lagrangian fibration $f$ is induced by~$L$ if $f$ is the map associated to the linear system of $L^{\tensor n}$ for all $n\gg 0$.
\end{definition}

\begin{remark}\label{remark fibration base locus}
It is convenient to use the term ``induced rational map of $L$'' to refer to the map defined by the linear system of a high enough \emph{multiple} of $L$. Then, however, it becomes crucial to require $L$ to be movable. Consider for example an elliptic K3 surface $f\colon S \to \P^1$ with a section, as in Example~\ref{example elliptic k3}. We denote by $\ell$ the class of a fiber and by $\sigma$ the class of a section. Then the linear system of $\ell+\sigma$ has $\sigma$ as a fixed component. According to Definition~\ref{definition lagrangian fibration}, the fibration $f$ is however not \emph{induced} by $\ell+\sigma$ as this is actually a big line bundle. Note that $2\ell + \sigma$ is big and nef.
\end{remark}

Clearly, a holomorphic Lagrangian fibration is always induced by a line bundle, more precisely, by the pullback of any ample bundle on the base. Let us recall that the pullback of a line bundle $M$ on $B$ along a rational map $f\colon X\ratl B$ is defined by taking a resolution of indeterminacy
\[
\xymatrix@R=3mm{
&\tilde X \ar[dl]_\pi\ar[dr]^{\tilde f}&\\
X \ar@{-->}[rr]^f && B\\
}
\]
and putting
\begin{equation}\label{eq rational pullback}
f^*M:=\left(\pi_*\tilde f^* M\right)^{\vee\vee}.
\end{equation}
In general, $f^*M$ is only a reflexive rank~$1$ sheaf. If $X$ is $\Q$-factorial, taking a (reflexive) tensor power of this construction gives a line bundle on $X$. Still, the question whether a rational Lagrangian fibration is induced by a line bundle is a bit subtle, as the following example shows.

\begin{example}
Assume $X$ is $\Q$-factorial, and let the line bundle $L$ on $X$ be given by the pullback of an ample $A\in \Pic(B)$ along $f\colon X\ratl B$ as in \eqref{eq rational pullback}. Then this need not induce the fibration in the sense of Definition~\ref{definition lagrangian fibration}. By reflexivity, it is clear that $f$ is the map associated to the linear system $f^*\abs{A}=\abs{L}$ (this last equality follows from $f$ having connected fibers), but multiples of $L$ might define a different map. Consider for example the singular elliptic K3 surface $\bar S \ratl \P^1$ from Example~\ref{example elliptic k3}. We discussed that $f^*\sO(1)$ is ample in that case. 

Another potential obstruction for a rational Lagrangian fibration $f\colon X \ratl B$ to be induced by a line bundle is non-$\Q$-factoriality. It would be interesting to have an explicit example $f$ where non-$\Q$-factoriality of $X$ obstructs the pullback of an ample line bundle on $B$ from being a ($\Q$-)line bundle. 
\end{example}

Observe that for a \emph{holomorphic} Lagrangian fibration $f\colon X \rightarrow B$, the pullback of an ample line bundle $A$ on $B$ satisfies $q_X(f^*A)=0$. This is a direct consequence of the Fujiki relation from Lemma~\ref{lemma bbf form}\eqref{lbf-5}. It turns out that (in the projective case) rational Lagrangian fibrations are not that far apart from holomorphic ones.

\begin{lemma}\label{lemma lmp}
Let $X$ be a projective primitive symplectic variety with $b_2(X) \geq 5$. Let $L$ be a $($movable\,$)$ line bundle on $X$ inducing a rational Lagrangian fibration and satisfying $q_X(L)=0$. Then there exist a  birational map $\phi\colon  X \ratl X'$ to a primitive symplectic variety $X'$ with $\Q$-factorial terminal singularities and a holomorphic Lagrangian fibration $f'\colon X'\to B$ such that the birational transform of\, $L$ is the pullback of an ample line bundle on $B$.
\end{lemma}
\begin{proof}
By taking a $\Q$-factorial terminalization of $X$, see \cite[Corollary~1.4.3]{BCHM10}, and pulling back the line bundle, we may assume that $X$ itself has $\Q$-factorial terminal singularities. By \cite[Theorem~1.2]{LMP22a}, there is a rational polyhedral fundamental domain for the action of the group of birational automorphisms of~$X$ on~$\Mov^+(X)$. Here, $\Mov^+(X)$ is defined as the convex hull of $\Movbar(X)\cap \Pic(X)_\Q$. From the proof, we deduce that there is a rational polyhedral cone $C^+\subset \Mov^+(X)$ containing $L$ and being contained in the nef cone of a birational model $X'$ of $X$. As both $X'$ and $X$ have $\Q$-factorial terminal singularities, they are isomorphic in codimension~$1$, and the pullback $L'$ of $L$ to $X'$ is still isotropic for the BBF form on $X'$. By assumption, the Kodaira--Iitaka dimension $\kappa(L)$ of $L$ is $n:=\dim X /2$, hence so is $\kappa(L')$. Since $L'$ is  $q_{X'}$-isotropic, its the numerical Kodaira dimension is also equal to $n$, so $L'$ is nef and abundant, and the claim follows from Kawamata's theorem \cite[Theorem~6.1]{Kaw85}; see also \cite[Theorem~1.1]{Fuj11}.
\end{proof}

\begin{remark}\label{remark rational lagrangian fibration}
If we have the MMP (that is, termination of flips) at our disposal, we can argue differently in the first part of Lemma~\ref{lemma lmp}.  Indeed, if $f$ is induced by a linear system $\abs{D}$ of a Cartier divisor $D$ on $X$, one can obtain $X'$ as in the definition by running a log-MMP for $(X,\Delta)$, where $\Delta$ is a general element in $\abs D$. Note that flips terminate if $X$ is smooth by \cite{LP16} or more generally if $X$ has hyperquotient singularities by~\cite{LMP22b}. In these cases, we can in particular drop the assumption $b_2(X)\geq 5$.

It is also likely that we can drop the projectivity assumption in Lemma~\ref{lemma lmp}. For K\"ahler  irreducible symplectic manifolds, for example, it also follows from the fact that the birational K\"ahler cone coincides with the (closure of the) fundamental exceptional chamber; see \cite[Theorem~1.5]{Mar11}. 
\end{remark}

\subsection{The SYZ conjecture}\label{section syz}

The SYZ conjecture is one of the most important conjectures about primitive symplectic varieties and is wide open in general. Note, however, that it is known in all known smooth examples; see Remark~\ref{remark syz smooth case} below. Before we state it, let us recall that given a Lagrangian fibration $f\colon X \rightarrow B$, the pullback of an ample class $A$ on $B$ satisfies $q_X(f^*A)=0$ and induces $f$ in the sense of Definition~\ref{definition lagrangian fibration}.

\begin{conjecture}[SYZ]\label{conjecture syz}
If $L$ is a non-trivial nef line bundle on a primitive symplectic variety $X$ with $q_X(L)=0$, then $L$ induces a holomorphic Lagrangian fibration.
\end{conjecture} 

\begin{remark}\label{remark syz smooth case}
In the smooth case, this conjecture is known for deformations of $\mathrm{K3}^{[n]}$ (Bayer--Macr\`i \cite[Theorem 1.5]{BM14mmp}; Markman \cite[Theorems 1.3 and 6.3]{Mar14}), for deformations of $K_n(A)$ (Yoshioka \cite[Proposition 3.38]{Yosh16}), and for deformations of the O'Grady examples ${\rm OG}_6$ (Mongardi--Rapagnetta \cite[Corollary 1.3 and 7.3]{MR21}) and ${\rm OG}_{10}$ (Mongardi--Onorati \cite[Theorem 2.2]{MO22}). 
\end{remark}

We will also need a rational version of the SYZ conjecture. 

\begin{conjecture}[Rational SYZ]\label{conjecture rational syz}
If $L$ is a non-trivial line bundle on a primitive symplectic variety $X$ with $L\in \Movbar(X)$ and $q_X(L)=0$, then $L$ is movable and induces a rational Lagrangian fibration.
\end{conjecture} 

\begin{lemma}\label{lemma qx positive is big}
Let $X$ be a primitive symplectic variety. If\, $M$ is a $\Q$-line bundle on $X$ with $q_X(M)>0$, then either $M$ or $M^\vee$ is big.
\end{lemma}
\begin{proof}
  By taking a suitable multiple, we may assume that $M$ is a line bundle. Then the proof is contained in that of \cite[Theorem~6.9]{BL22}. Notice that in the notation of \textit{loc.~cit.}, the assertion that $L$ is big is incorrect\footnote{We thank the referee for bringing this inaccuracy to our attention.} and should be replaced by saying  that $L$ \emph{or $L^\vee$} is big. The hypothesis $q_X(L)>0$ only implies that $\chern_1(L)$ \emph{or its negative} lies in the positive cone. On p.~251 in the proof of \cite[Theorem~6.9]{BL22}, it is tacitly assumed that this is the case for~$\chern_1(L)$ itself.
\end{proof}

With the corrections to \cite[Theorem~6.9]{BL22} made above, the claim of Lemma~\ref{lemma qx positive is big} can alternatively be reduced to the statement of \textit{loc.\ cit.}, as was done in \cite[Lemma~4.7]{LMP22a}.

We prove the following variation of \cite[Proposition~5.6]{LMP22a}.

\begin{proposition}\label{proposition lmp}
Let $X$ be a projective $\Q$-factorial primitive symplectic variety and $D\geq 0$ a big $\Q$-divisor on $X$. If\, $D = P(D) + N(D)$ denotes its Boucksom--Zariski decomposition in the sense of\, \cite[Theorem~1.1]{KMPP19}, then $P(D)$ is movable.
\end{proposition}
\begin{proof}
The proof is essentially that of \cite[Proposition~5.6]{LMP22a}: In \textit{loc.\ cit.}, the assumption is that $D\in C_X$, which is only used to ensure that $D$ is big. Hence, the argument remains valid in our situation. Note that the $\mathbb{Q}$-factoriality assumption is missing in \cite[Proposition~5.6]{LMP22a} but is needed when using \cite[Corollary~1.4.2]{BCHM10}. 
\end{proof}

In the context of the rational SYZ conjecture, the following lemma will be useful.

\begin{lemma}\label{lemma movbar and sections imply movable}
Let $X$ be a primitive symplectic variety, and let $L$ be a line bundle on~$X$ with Kodaira dimension $\kappa(L)>0$. Then, $L \in \Movbar(X)$ if and only if\, $L$ is movable. 
\end{lemma}
\begin{proof}
For the non-trivial direction, we distinguish two cases. 

\begin{enumerate}[wide, label={{\it Case}~\rm\arabic*:}, ref=\arabic*]
\item \emph{$L$ is big.}\label{case1}

In this case, $X$ has to be projective by \cite[Theorem~1.6]{Nam02}. By \cite[Corollary~1.4.3]{BCHM10}, we can assume that $X$ is $\Q$-factorial. Since $L$ is a limit of movable line bundles, it is $q_X$-nef in the sense of \cite[Definition~3.2]{KMPP19}. Together with the uniqueness statement in \cite[Theorem~1.1]{KMPP19}, this implies that $L$ is equal to the positive part in its Boucksom--Zariski decomposition. Therefore, since $L$ is big, it is also movable by Proposition~\ref{proposition lmp}. 

\item \emph{$L$ is not big.}\label{case2}

Up to replacing $L$ by a multiple, we may write $\abs{L} = \abs{M} + F$, where $M$ is the movable part, $\dim |M|\geq 1$ (or, equivalently, $M\neq 0$), and $F$ is the fixed part. Since $M\neq 0$, we may assume that $M$ and $F$ are not proportional, for otherwise there is nothing to prove. We will proceed in several steps.

\begin{enumerate}[wide, label={{\it Step}~\rm\arabic*:}, ref=\arabic*]
\item We observe that $q_X(L)=0$. 

Indeed, as $L\in\Movbar(X)$, we have $q_X(L) \geq 0$. If $q_X(L)>0$, then $L$ would be big by Lemma~\ref{lemma qx positive is big}, contradicting the assumption of Case~\ref{case2}. Note that the dual of the non-trivial effective line bundle $L$ cannot be big.

\item\label{step2} We show that $q_X(M)=0$.

As $M$ is movable, we have $q_X(M) \geq 0$. If $q_X(M)>0$, then, again by Lemma~\ref{lemma qx positive is big}, the line bundle $M$ would be big. Since $F$ is effective, we have $\kappa(M)\leq \kappa(L)$. But $L$ is not big, so we have a contradiction.

\item\label{step3} We show that $q_X(M,F)=0$.

Since $M$ is movable and $F$ is effective, $q_X(M,F)\geq 0$. Suppose $q_X(M,F)>0$. Thus, $q_X(L+M)=2q_X(M,L)=2q_X(M,F)>0$, so  $L+M$ is big by Lemma~\ref{lemma qx positive is big} and the fact that $L+M$ is effective. But then,  $L+M+F =2 L$ would also be big, as it is the sum of a big and an effective line bundle. This again contradicts the fact that $L$ is not big.

\item We show that $F=0$, thereby completing the proof. 
\end{enumerate}
By Steps~\ref{step2} and~\ref{step3}, we have $q_X(F)=q_X(M+F)=q_X(L)=0$ and consequently $q_X(F,L)=q_X(F)+q_X(F,M)=0$ so that $H^{1,1}(X,\R)$ would contain the isotropic plane spanned by $L$ and $F$. This is a contradiction to the signature of $q_X$ on $H^{1,1}(X,\R)$ being equal to $(1,h^{1,1}(X)-1)$ by Lemma~\ref{lemma bbf form}.\qedhere
\end{enumerate}
\end{proof}

We note that for line bundles $L \in \Movbar(X)$, bigness is actually equivalent to $q_X(L)>0$ (and hence to having positive top self-intersection, just as for nef line bundles). 

\begin{lemma}\label{lemma movbar and isotropic implies not big}
Let $X$ be a primitive symplectic variety with $\Q$-factorial terminal singularities, and let $L$ be a line bundle on~$X$ with $q_X(L)=0$. If\, $L \in \Movbar(X)$, then $L$ is not big.
\end{lemma}
\begin{proof}
Suppose that $L$ is big. Then, as before, $X$ has to be projective by \cite[Theorem~1.6]{Nam02}. Arguing as in the proof of \cite[Proposition~5.6]{LMP22a} (\textit{cf.} Proposition~\ref{proposition lmp} and its proof, which explain why this is possible as soon as $L$ is big), we obtain a birational map \mbox{$\phi\colon X \ratl X'$} to a normal projective variety $X'$ such that 
\begin{itemize}
	\item the variety $X'$ has $\Q$-factorial terminal singularities; 
	\item the map $\phi$ is an isomorphism in codimension~$1$; in particular, $X'$ is again primitive symplectic; 
	\item the pushforward $\phi_*L$ is nef.
\end{itemize}
In particular, $q_{X'}(\phi_*L)>0$. But $\phi_*$ is a Hodge isometry by \cite[Theorem~4.2]{LMP22a}. This contradicts the hypothesis $q_X(L)=0$.
\end{proof}

We will need the following result on non-algebraic primitive symplectic varieties.

\begin{lemma}\label{lemma algebraic dimension}
Let $X$ be a primitive symplectic variety of dimension $2n$ which is not projective. Then any line bundle $L$ on $X$ has $\kappa(L) \leq n$.
\end{lemma}
\begin{proof}
The proof is essentially the same as that of \cite[Theorem~3.6]{COP10}. Suppose \mbox{$f\colon X \ratl B$} is a dominant meromorphic map to a normal projective variety $B$ associated to the linear system of a line bundle $L$ on $X$. We assume that $\dim B > n$. 

Note that $\dim B < 2n$, as $X$ is not projective. By subtracting the fixed part, we may assume that $L$ is movable, in particular, $q_X(L) \geq 0$. Since $X$ is non-projective, $L$ cannot be big; \textit{cf.} the proof of Lemma~\ref{lemma movbar and sections imply movable}, Case~\ref{case1}. Hence, we must have $q_X(L)=0$ by Lemma~\ref{lemma qx positive is big}.

Let us choose a resolution of indeterminacy by a compact K\"ahler manifold $Y$ and obtain a diagram
\[
\xymatrix@R=2ex@C=3em{
&Y \ar[dl]_\pi \ar[dr]^g&\\
X \ar@{-->}[rr]_f && B.\\
}
\]
Since $L$ is movable, we have that $\pi^*L = g^*A + E$, where $A$ is an ample bundle on $B$ and $E$ is effective and $\pi$-exceptional. Let $\kappa$ be a K\"ahler class on $X$. We will show that for non-negative integers $0\leq a \leq c$ with $n < c \leq 2n$, we have
\begin{equation}\label{eq cop induction}
(g^*A)^a\cdot(\pi^*L)^{c-a}\cdot(\pi^*\kappa)^{2n-c} = 0.
\end{equation}
The proof proceeds by induction on $a$. The base case is equivalent to $L^c\cdot \kappa^{2n-c}=0$. This follows from $L^c=0$, which in turn is a consequence of $q_X(L)=0$ and \cite[Proposition~5.11]{BL22}. Suppose the statement is proven for $a<c$. Then
\[
\begin{aligned}
0&=(g^*A)^a\cdot(\pi^*L)^{c-a}\cdot(\pi^*\kappa)^{2n-c}\\
&= (g^*A)^{a+1}\cdot (\pi^*L)^{c-(a+1)}\cdot (\pi^*\kappa)^{2n-c}  + E\cdot (g^*A)^a\cdot(\pi^*L)^{c-(a+1)}\cdot(\pi^*\kappa)^{2n-c}. \\
\end{aligned}
\]
Since the classes $g^*A$, $\pi^*L$, and $\pi^*\kappa$ are nef and $E$ is effective, both terms have to vanish individually, and \eqref{eq cop induction} follows. We specialize to $a=c=\dim B$. Then, $(g^*A)^a$ is a positive multiple of the class $F$ of a general fiber of $g$, and we obtain $F\cdot(\pi^*\kappa)^{2n-c}=0$. By the projection formula, we find $(\pi_*F)\cdot\kappa^{2n-c}=0$, which is absurd because $\kappa$ is K\"ahler and $\pi_*F$ is a non-trivial effective cycle of codimension $c$.
\end{proof}

Lemma~\ref{lemma algebraic dimension} also shows that the algebraic dimension of $X$ is at most $n$; \textit{cf.} \cite[Theorem~3.6]{COP10}.

It seems worthwhile clarifying the relation between the rational and the holomorphic version of the SYZ conjecture.

\begin{lemma}\label{lemma rational syz versus syz}
Let $X$ be a primitive symplectic variety of dimension $2n$ with \mbox{$\Q$-factorial} terminal singularities. If $b_2(X) \geq 5$ or $X$ is smooth, then the following statements are equivalent: 
\begin{enumerate}
    \item\label{lrsvs-1} All primitive symplectic varieties locally trivially deformation equivalent to $X$ satisfy the SYZ conjecture.
    \item\label{lrsvs-2} All primitive symplectic varieties locally trivially deformation equivalent to $X$ satisfy the rational SYZ conjecture.
\end{enumerate}
\end{lemma}
\begin{proof}
Assume that~\eqref{lrsvs-2} holds, and let $L$ be a non-trivial nef line bundle on $X$ with $q_X(L)=0$. As $X$ is arbitrary, it suffices to prove the existence of a Lagrangian fibration on $X$ itself. By assumption, global sections of $L$ give rise to a rational Lagrangian fibration $f\colon X \ratl B$. By a standard argument, the Fujiki relation (Lemma~\ref{lemma bbf form}\eqref{lbf-5}) implies that $L^{n+1}=0$ in cohomology while $L^n\neq 0$. 
If $X$ is projective, it follows from Kawamata's semi-ampleness theorem \cite[Theorem~6.1]{Kaw85} that $f$ is regular (as $L$ is nef and abundant). For non-projective $X$, we use \cite[Theorem~5.5]{Nak87} instead.

Now suppose  that~\eqref{lrsvs-1} holds, and let $L\in\Movbar(X)$ be a non-trivial line bundle with $q_X(L)=0$. As above, it suffices to show that $L$ is movable and induces a rational Lagrangian fibration. By Lemma~\ref{lemma movbar and sections imply movable}, it suffices to prove that $\kappa(L)=n$. To this end, we will first show that some locally trivial deformation $(X_t,L_t)$ of the pair $(X,L)$ admits a regular Lagrangian fibration. 

We will first deal with the case $b_2(X)\geq 5$. Choose $(X_t,L_t)$ such that $X_t$ is projective. By \cite[Corollary~6.11]{BL22}, such pairs are dense in the local Kuranishi space of $(X,L)$. Let us first assume that $L_t \in \Mov^+(X_t)$. Recall that $\Mov^+(X_t)$ was defined as the convex hull of $\Movbar(X_t)\cap \Pic(X_t)_\Q$, so this condition is equivalent to $L_t \in \Movbar(X_t)$. Now one argues as in the proof of Lemma~\ref{lemma lmp} to see that there is a birational model $X_t'$ of $X_t$ such that the line bundle $L_t'$ on $X_t'$ corresponding to $L_t$ is nef. By our assumption, $X_t'$ satisfies the SYZ conjecture; hence~$L_t$ induces a rational Lagrangian fibration on~$X_t$.

Let us now consider the case where $L_t$ is not in the closure of the movable cone. Then, by \cite[Proposition~5.8 and Theorem~5.12]{LMP22a},  the group generated by reflections in prime exceptional divisors contains an element that maps $L_t$ to some $M_t\in\Mov^+(X_t)$. By \cite[Theorem~3.10]{LMP22a}, the pair $(X_t, M_t)$ is a deformation of $(X_t,L_t)$, so we are back in the previous case. 

Now we use that Lagrangian fibrations locally deform over their Hodge locus by \cite[Theorem~A.2]{EFGMS25} and infer from the semi-continuity of $h^0(X_t,L_t^{\tensor  n})$ that $L$ has Kodaira dimension at least~$n$. If $X$ is non-projective, we conclude by Lemma~\ref{lemma algebraic dimension}. If $X$ is projective, we argue as for~$X_t$ above and deduce that there is a birational model $X'$ of $X$ such that the isotropic line bundle $L'$ on $X'$ corresponding to $L$ is nef. In this case, $\kappa(L)=\kappa(L')$. The Kodaira dimension is always bounded above by the numerical dimension of a nef line bundle, so $\kappa(L')\leq \nu(L')$. Moreover, as shown in the proof of \eqref{lrsvs-2}$\Rightarrow$\eqref{lrsvs-1}, the condition $q_X(L')=q_X(L)=0$ implies $\nu(L')=n$.
In either case, we have $\kappa(L)=n$, and this completes the argument for $b_2(X)\geq 5$.

The proof for smooth $X$ is similar. At a very general point $(X_t,L_t)$ of the Kuranishi space of $(X,L)$, the Picard group of $X_t$ is generated by $L_t$. We deduce that such an $L_t$ is nef by \cite[Proposition~28.2]{GHJ03}. Recall that the nef cone is by definition the closure of the K\"ahler cone so that this statement is non-trivial. By our assumption, $L_t$ induces a Lagrangian fibration on $X_t$, and we argue by semi-continuity as before.
\end{proof}

Note that the proof of \eqref{lrsvs-2}$\Rightarrow$\eqref{lrsvs-1} had no assumptions on $b_2$ or singularities and did not resort to deformations. The smoothness hypothesis can be relaxed to having quotient singularities with $\codim_X X^\sing \geq 4$ by \cite[Corollary~5.6]{Men20}. We believe that the codimension assumption can be dropped if one copies Menet's argument, replacing arbitrary deformations by locally trivial ones.

\section{Hyperbolicity}\label{section hyperbolicity}

Here, we recall some classical hyperbolicity notions that can be found in \cite{Ko76} and \cite{Br78}. 

\begin{definition}\label{definition hyperbolic}
Let $X$ be a complex variety. The {\em Kobayashi pseudometric} on $X$ is the maximal pseudometric $d_X$ such that all holomorphic maps $f \colon (D,\rho) \rightarrow (X,d_X)$ are  distance decreasing, where $(D,\rho)$ is the disk with the Poincar\'e metric. A variety $X$ is {\em Kobayashi hyperbolic} if $d_X$ is a metric.
\end{definition}

One immediately sees that the complex line $\C$ is not Kobayashi hyperbolic. In fact, the Kobayashi pseudometric of $\C$ vanishes identically. Therefore, the existence of an entire curve (that is, a non-constant holomorphic map from the complex line) implies Kobayashi non-hyperbolicity. The converse holds for compact manifolds.

\begin{theorem}[Brody]\label{Brody theorem}\leavevmode
\begin{enumerate}
    \item\label{Bt-1} Let $X$ be a compact complex space. Then $X$ is Kobayashi non-hyperbolic if and only is there exists an entire curve $\mathbb C \rightarrow X$. 
    \item\label{Bt-2} The Kobayashi non-hyperbolicity property is preserved on taking limits.
\end{enumerate}
\end{theorem}
\begin{proof}
The first statement is due to Brody for smooth $X$, see \cite[Theorem~4.1]{Br78}, and the argument essentially goes through in the singular case, see \textit{e.g.}~Lang's book \cite[Theorem~III.2.1]{Lang}. 
The second statement is \cite[Theorem~3.1]{Br78} in the smooth case. In the singular case, all that is needed is the notion of a length function on a complex space as in \cite[Chapter~0]{Lang}. Using compactness, one can argue that a limit of disks with increasing radii gives a Brody curve; see \textit{e.g.}~\cite[Lemma~2.8]{BKV20}.
\end{proof}

A variety admitting no entire curve is sometimes called \emph{Brody hyperbolic}. Brody's theorem thus says that for compact complex varieties, Brody hyperbolicity coincides with Kobayashi hyperbolicity.

\begin{remark}\leavevmode
\begin{enumerate}
\item Note that Kobayashi hyperbolicity always implies Brody hyperbolicity, even without the hypothesis of compactness. Indeed, as soon as there is an entire curve, the Kobayashi distance between points in its image has to be zero. 
\item While Brody's theorem tells us that the limit of non-hyperbolic compact manifolds is non-hyperbolic, the limit of \emph{hyperbolic} compact manifolds can be either hyperbolic or non-hyperbolic. For an example of hyperbolic manifolds specializing to a non-hyperbolic one, we consider a generic family of degree~$d$ hypersurfaces in $\P^n$. For big enough $d$, these are hyperbolic by the main theorem of \cite{Bro17}. They specialize, however, to the Fermat hypersurface given by the polynomial $x_0^d + x_1^d + \cdots + x_n^d$,  which contains a line as soon as $n\geq 3$.
    \item If we allow singular fibers, it is even easier to obtain a (Brody, hence also Kobayashi) non-hyperbolic variety as the limit of hyperbolic ones. Take for example a family of genus $2$ curves that degenerate to a nodal elliptic curve. There, the general fiber is hyperbolic, while the special fiber is not.
\item If we drop the compactness requirement, then Brody's theorem fails; see \cite[Example~3.6.6]{Kob98} for an example of a Kobayashi non-hyperbolic manifold with no entire curves.
\item Non-hyperbolicity is closed (with respect to the Euclidean topology) in families of singular varieties by Theorem~\ref{Brody theorem}\eqref{Bt-2}. 
\end{enumerate}
\end{remark}

\begin{lemma}\label{lemma properties hyperbolicity}
All varieties are assumed to be compact. Let $P$ be one of the properties ``is non-hyperbolic'' or ``satisfies $d_X=0$,'' where $d_X$ is the Kobayashi pseudometric. Then, the following hold: 
\begin{enumerate}
\item\label{lph-1} Holomorphic maps $f\colon X \to Y$ are distance decreasing for the Kobayashi metric. 
\item\label{lph-2} If $f\colon X\to Y$ is finite, then $Y$ has property $P$ if\, $X$ does. For finite \'etale morphisms, the converse holds.
\item\label{lph-3} Let $X, Y$ be compact varieties and $f\colon X\ratl Y$ a dominant meromorphic map. If $d_X = 0$, then $d_Y= 0$ as well.
\item\label{lph-4} If\, $X=X_1\times \cdots \times X_n$ and $X_i$ has property $P$ for all  $i=1,\ldots,n$, then $X$ has property $P$. 
\item\label{item bimeromorphic to smooth} If $f\colon Y\to X$ is a bimeromorphic morphism onto a smooth variety $X$ with $d_X\leq \veps$ for some $\veps \geq 0$, then  $d_Y \leq \veps$ as well.
\end{enumerate}
\end{lemma}
\begin{proof}
The first four items are standard. For the last item, we use that for a Zariski closed subset $V \subset X$ of codimension at least $2$, we have
\begin{equation}\label{eq kobayashi codimension two}
\left(d_X\right)\vert_{X \setminus V} = d_{X \setminus V}; 
\end{equation}
see \cite[Theorem~3.2.19]{Kob98}.
\end{proof}

Note that if in~\eqref{lph-4} the product $X$ has vanishing Kobayashi metric, then so does every single $X_i$. However, the product of a hyperbolic and a non-hyperbolic compact manifold is non-hyperbolic, and therefore the statement of Lemma~\ref{lemma properties hyperbolicity}\eqref{lph-3} is not ``if and only if.''

\begin{example}\label{example ariyan}
Given a surjective morphism $f\colon X \to B$ of complex varieties such that $d_X$ vanishes,  $d_B$ and the Kobayashi pseudometric of the fibers also vanish. The converse, however, is false due to the presence of multiple fibers, as the following example shows. Let $C$ be a genus~$2$ curve and $E$ an elliptic curve. Consider the $\mu_2$-action on $C \times E$ given by the hyperelliptic involution $\iota$ on C and translation by a $2$-torsion point on~$E$. As the action is free, its quotient $X:=C\times E / \mu_2$ is smooth and the quotient morphism $\pi\colon C \times E \to X$ is finite \'etale. In particular, $d_X$ does not identically vanish by Lemma~\ref{lemma properties hyperbolicity}. However, the base and the fibers of the morphism $X \to C/\iota\isom \P^1$ have vanishing Kobayashi distance. Indeed, if $\Sigma \subset C/\iota$ is the ramification locus, the fiber over a point in the complement of $\Sigma$ is an elliptic curve (namely $E$) and the fibers over points of $\Sigma$ are isomorphic to $\P^1$ with multiplicity $2$.
\end{example}

\begin{example}\label{example contraction}
Let $C \subset \P^2\subset \P^3$ be a curve of genus at least~$2$, and let $X \subset \P^3$ be the cone over $C$ with vertex $ v\notin \P^2$. Let $\pi\colon Y \to X$ be the blowup in $v$. Then, $\pi$ is a resolution of singularities, and the exceptional divisor $E$ is a section of a $\P^1$-bundle $f\colon Y\to C$. As $f$ is distance decreasing, we see that $d_Y$ cannot vanish identically. On the other hand, $X$ is rationally chain connected; hence $d_X \equiv 0$. This example shows that the vanishing of the Kobayashi pseudometric is not a birational invariant. Moreover, the quasi-projective variety $Y\ohne E\isom X\ohne \{v\}$ cannot have vanishing Kobayashi pseudometric. This is in stark contrast with the situation for smooth varieties, where the Kobayashi pseudometric is determined by its restriction outside a codimension $2$ subset; see \cite[Theorem~3.2.19]{Kob98}. 
\end{example}

One can still wonder whether the Kobayashi pseudometric is determined by its restriction outside a codimension $2$ subset under some assumptions on the singularities. In concrete terms, consider the following. 

\begin{question}\label{qu1}
Let $X$ be a complex variety with log-terminal (or, more generally, rational) singularities.
\begin{enumerate}
\item\label{qu1-1} Let $V \subset X$ be a Zariski closed subset of codimension at least~$2$. Is it true that \eqref{eq kobayashi codimension two} still holds?
\item\label{qu1-2} Is it true that if $\pi\colon Y \to X$ is a resolution of singularities, then the vanishing of $d_X$ implies the vanishing of $d_Y$?
\end{enumerate}
\end{question}

Note that a positive answer to~\eqref{qu1-1} implies~\eqref{qu1-2}. A positive answer to~\eqref{qu1-2} in full generality would simplify our argument. In our main result, we actually make heavy use of birational modifications; see Section~\ref{section non-hyperbolicity results}. However, we are not affected by the above questions as we will have some stronger geometric input.

\section{Almost holomorphic maps and Campana's theorem}\label{section almost holomorphic}

This section surveys basic notions and results on almost holomorphic maps, the most important of which is undoubtedly Campana's theorem which allows us to produce almost holomorphic maps out of covering families of cycles; see Theorem~\ref{theorem campana}. There are no new results, only Theorem~\ref{theorem glr} is a slight adaption from a result of \cite{GLR13} on irreducible symplectic manifolds to the singular case. We begin by collecting basic results about cycle spaces.

\subsection{Cycle spaces}\label{section cycle spaces}

Let $X$ be a compact complex space. We denote by $\scrB(X)$ Barlet's space of cycles on $X$; see \cite{Bar75}. For a subspace $\gothF \subset \scrB(X)$, we denote by $(F_t)_{t\in\gothF}$ the analytic family of cycles parametrized by $\gothF$. Here, $F_t$ is the cycle corresponding to $t\in\gothF$. If $F$ is a cycle on $X$, we denote its support by $\abs F \subset X$. We will usually drop the word \emph{analytic} and just speak of a family of cycles.

If $(F_t)_{t\in\gothF}$ is a family of cycles, we denote by
\[
\Gamma_\gothF := \{(t,x) \in \gothF \times X \mid x \in \abs{F_t} \} \subset \gothF \times X
\]
its graph, which is an analytic subset in $\gothF\times X$ by \cite[Theorem~VIII.2.7]{SCV7}. We say that $\gothF$ is a \emph{covering} family of cycles if

\[
\bigcup_{t\in \gothF} \abs{F_t} = X.
\]

The actual definition of an analytic family of cycles is a bit involved; see \cite[d\'efinition fondamentale, p.~33]{Bar75}, but we will not need it here. The Barlet space is the universal object classifying analytic families of cycles in the sense that every such family is obtained by pullback along a uniquely determined classifying map from the universal family of cycles. A very useful tool for obtaining families of cycles is the following proposition taken from \cite[Proposition~VIII.2.20]{SCV7}. 

\begin{proposition}\label{proposition family of cycles}
Let $X$ and $S$ be irreducible compact complex spaces. Then, there is a one-to-one correspondence between
\begin{enumerate}
    \item meromorphic maps $S\ratl \scrB(X)$ and
    \item pure-dimensional, $S$-proper cycles $F$ on $S \times X$.
\end{enumerate}
\end{proposition}

\subsection{Almost holomorphic maps}

\begin{definition}\label{definition almost holomorphic}
For a meromorphic map $f\colon X \ratl B$ and a subset $U\subset B$, we denote by $f^{-1}(U)$ the set of points from the domain of definition of $f$ that map to $U$. The \emph{fiber} of $f$ over $b\in B$ is the closure of $f^{-1}(b)$. A dominant meromorphic map $f\colon X \ratl B$ between compact complex varieties is called \emph{almost holomorphic} if there is a dense open subset $U\subset B$ such that $f\vert_{f^{-1}(U)}\colon f^{-1}(U)\to U$ is holomorphic and  proper.
\end{definition}

Note that being almost holomorphic can also be phrased by saying that the fibers of $f$ are pairwise disjoint over a dense open set in the target.

An important theorem due to Campana allows us to produce many almost holomorphic maps out of (covering) families of cycles. We need to introduce some terminology in order to formulate it. Let $X$ be a compact K\"ahler space,\footnote{Actually, of Fujiki class is sufficient here.} and suppose we are given a family $\{F_t\}_{t\in \gothF}$ of cycles, where $\gothF \subset \scrB(X)$ is a closed subspace of the Barlet space. Then, one can define an equivalence relation on $X$ as follows. 

\begin{definition}\label{definition f-connected}
Two points $x,y\in X$ are \emph{$\gothF$-equivalent} (or simply \emph{equivalent} if the family $\gothF$ is clear from the context) if they can be connected by a chain of cycles in $\gothF$ or if $x=y$. By definition, being \emph{connected by a chain of cycles} in $\gothF=\{F_t\}_{t\in\gothF}$ means that there exist finitely many points $x=x_1, x_2, \ldots, x_{n+1}=y$ and $t_1,\ldots,t_n \in \gothF$ such that $x_i,x_{i+1}\in \abs{F_{t_i}}$ for all $i=1,\ldots,n$. We write $x \sim_\gothF y$ (or simply $x \sim y$) to express that $x$ and $y$ are equivalent.
\end{definition}

It is clear that $\gothF$-equivalence is an equivalence relation. Observe that every $x\in X$ can be connected to itself by a chain of cycles in $\gothF$ if and only if the family is covering. We are now able to state Campana's theorem; see~\cite[Th\'eor\`eme~1]{Cam81}. 

\begin{theorem}[Campana]\label{theorem campana}
Let $X$ be a compact complex space which is globally and locally irreducible. Let $\gothF \subset \scrB(X)$ be a closed subspace, let $(F_t)_{t\in\gothF}$ be the corresponding family of cycles, and assume that for a general point $t\in \gothF$, the cycle $F_t$ is integral. Then there is an almost holomorphic map $f\colon X \ratl B$ such that general fibers of~$f$ are equivalence classes for the relation of\, $\gothF$-equivalence. 
\end{theorem}

Note that the statement of the theorem is trivial in case $\gothF$ is not a covering family of cycles. In Campana's original result, the subspace $\gothF$ was assumed to be irreducible, but this assumption can be removed; see \cite[Theorem~1.1]{Cam04}. An algebraic version of Campana's theorem has been obtained by Koll\'ar \cite[Theorem~2.6]{Kol87}; see also \cite[Chapter~5]{Deb01}.

\begin{remark}
The space $B$ from Theorem~\ref{theorem campana} is constructed in \cite{Cam81} as a subspace of the Barlet space. Therefore, it can be chosen K\"ahler (respectively, of Fujiki class) if $X$ is K\"ahler (respectively, of Fujiki class); see \cite[Th\'eor\`eme~2]{Var86} or \cite[Theorem~4']{Var89} in the K\"ahler case and \cite[Corollaire~3]{Cam80} for spaces of Fujiki class.
\end{remark}

\subsection{Almost holomorphic Lagrangian fibrations}

The following theorem is a slight adaption of \cite[Lemma~6.6]{GLR13} for primitive symplectic varieties. Some special attention has to be paid to $\Q$-factoriality and to ``horizontal'' singularities. We include a sketch of the argument for convenience.

\begin{theorem}\label{theorem glr}
Let $X$ be a projective primitive symplectic variety, $B$ a projective variety, and let $f\colon X \ratl B$ be a dominant almost holomorphic map with $0 < \dim B < \dim X=2n$. Then there is a diagram
\[
\xymatrix{
X \ar@{-->}[r]^\phi\ar@{-->}[d]_f & X' \ar[d]^{f'}\\
B \ar@{-->}[r]^\psi & B'\rlap{,}
}
\]
where $X'$ is a primitive symplectic variety, $B'$ is a normal projective variety, $f'\colon X'\to B'$ is a $($holomorphic\,$)$ Lagrangian fibration, $\phi$ and $\psi$ are birational, and $\phi$ is holomorphic in a neighborhood of the general fiber of $f$. In particular, $\dim B = n$ and $f$ is a rational Lagrangian fibration.
\end{theorem}

\begin{proof}
Replacing $X$ by a $\Q$-factorialization, see \cite[Corollary~1.4.3]{BCHM10}, we may assume that $X$ itself has only $\Q$-factorial singularities. In this case, we may define $D:=f^*A$ for some very ample Cartier divisor~$A$ on $B$. We choose a rational number $\delta>0$ small enough such that the pair $(X, \Delta)$ with $\Delta:=\delta D$ is klt. Note that this is always possible, as a primitive symplectic variety has canonical singularities. We choose a dense open $U\subset B$ such that $f\vert_{f^{-1}(U)}\colon f^{-1}(U) \to U$ is holomorphic and proper. Then we resolve the indeterminacy of the linear series $\abs{d\Delta}$, where $d$ is such that $d\Delta$ is Cartier, see \cite[Definition~9.1.11]{Laz04}, by a proper modification $p\colon \wt X \to X$. This implies that 
\[
p^*\abs{d\Delta} = \abs M + G, 
\]
where $M$ is free and $G$ is a fixed component. We can choose $d$ big enough so that the above equality holds for all multiples of $d\Delta$, \textit{i.e.}~$G$ becomes the stable base locus of $p^*\abs{d\Delta}$. This is parallel to \cite[Section~6.3.1]{GLR13}, but note that unlike there, $p$ is not necessarily an isomorphism over $f^{-1}(U)$ due to the singularities of $X$. Nevertheless, the proof of \cite[Lemma~6.6]{GLR13} goes through with minor changes. We will comment on where we deviate from that proof. 

First, we define $B'$ as the target of the map $\wt f\colon \wt X \to B'$ given by the linear system of $mM$ for big enough~$m$. Hence, $B'$ is a normal projective variety, and we obtain a commutative diagram
\[
\xymatrix{
X \ar@{-->}[d]_f & \wt X \ar[l]_p \ar[d]^{\wt f}\\
B \ar@{-->}[r]^\psi & B'\rlap{,}
}
\]
where $\psi$ birational and an isomorphism on $U\subset B$ by construction. As $\psi\circ f$ is also almost holomorphic, we may assume $B=B'$. Next, one considers the canonical bundle formula for the pair:
\[
K_{\wt X} + p^{-1}_*\Delta = p^*(K_X + \Delta) + F - E, 
\]
where $F, E$ are effective divisors supported on the exceptional locus of $p$. Moreover, $\floor*{E}=0$ as $(X,\Delta)$ is klt. As $X$ has canonical singularities, $E$ does not dominate $B$. We set $\wt\Delta:=\Delta + E$, and, after possibly shrinking $\delta$ further, we may assume that the pair $(\wt X, \wt \Delta)$ is klt. 

By adjunction (and up to shrinking $U$ if necessary), a fiber of $f$ over a point in $U$ has trivial canonical bundle and thus canonical singularities. Since $E$ does not meet the general fiber, the restriction of $K_{\wt X}+ \wt\Delta$ to the general fiber of $f\circ p$ is semi-ample, so in particular the pair $(\wt X, \wt \Delta) \times_B U$ has a good minimal model over~$U$ (namely $(f^{-1}(U),\Delta\vert_{f^{-1}(U)})$). By \cite[Theorem~1.1]{HX13}, the pair $(\wt X,\wt \Delta)$ has a good minimal model $(X',\Delta')$ over $B$. Note that unlike in \cite[Section~6.3.2]{GLR13}, the  pair $(\wt X, \wt \Delta) \times_B U$ is not necessarily itself a good minimal model, but the result of Hacon--Xu still applies. Let $\chi\colon \wt X \ratl X'$ denote the corresponding birational map for which $\Delta'=\chi_* \wt\Delta$. As in \cite[Theorem~4.4]{Lai11} and \cite[Section~6.3.3]{GLR13}, one shows that $(X',\Delta')$ is actually a minimal model of $(\wt X,\wt \Delta)$ (\textit{i.e.}, not only over $B$), that $X'$ has trivial canonical bundle (see \cite[Claim~6.9]{GLR13}) and hence is primitive symplectic, and that $\Delta'$ is semi-ample and induces the morphism $f'\colon X'\to B$. 

It remains to show that the birational map $\phi:=\chi\circ p^{-1}\colon X\ratl X'$ is an isomorphism in a neighborhood of the general fiber of $f$. Let us point out that this was automatic in \cite{GLR13}. By Theorem~\ref{theorem matsushita}, $f'$ is a smooth morphism with abelian fibers over a dense open set $V\subset B^\reg$ in the regular locus of the base. Let us consider a resolution of indeterminacy of $\phi\colon X'\times_BV \ratl X\times_B V$ over $V$ as in the diagram below. This can for example be obtained as a resolution of the closure of the graph of $\phi$. 
\[
\xymatrix{
&W \ar[dr]^{q'}\ar[dl]_q&\\
X\times_B V \ar@{-->}[rr]^{\phi} \ar[dr]_f && X'\times_BV \ar[dl]^{f'}\\
&V&\\
}
\]
Then every curve that is contracted by $q$ is mapped to a fiber of $f'$ under $q'$. But $X$ has canonical singularities; hence the fibers of $q$ are rationally chain connected by \cite[Corollary~1.5]{HM07}. Thus, they have to be contracted by $q'$ as well, and this makes $\phi$ holomorphic on $X\times_BV$. Since $X$ and $X'$ have trivial canonical divisor and $X'\times_BV$ is smooth, the morphism $\phi\colon X\times_BV\to X'\times_BV$ is an isomorphism, as claimed.
\end{proof}

Combining Theorems~\ref{theorem campana} and~\ref{theorem glr}, we immediately obtain an almost holomorphic version of Matsushita's theorem.

\begin{theorem}\label{theorem almost holomorphic matsushita}
Let $X$ be a projective primitive symplectic variety and $\gothF \subset \scrB(X)$ be a closed subspace whose general point corresponds to an integral cycle. Let \mbox{$f\colon X \ratl B$} be an almost holomorphic map whose general fibers are $\gothF$-equivalence classes. If\, $\dim B \notin\{0,\dim X\}$, then $f$ is a rational Lagrangian fibration. \qed
\end{theorem}

\section{Non-hyperbolicity of holomorphic symplectic varieties}\label{section non-hyperbolicity results}

The purpose of this section is to prove our main result, Theorem~\ref{theorem main introduction}, which is concerned with the non-hyperbolicity of holomorphic symplectic varieties and vanishing of the Kobayashi pseudometric. We will indeed prove a slightly stronger version, see Theorem~\ref{theorem main}, and in order to formulate it, we recall the notion of the rational rank of a period. Let $\Lambda$ be a lattice of signature $(3,n)$, and consider the period domain
\begin{equation}
    \Omega_\Lambda:=\{[x]\in \P(\Lambda\tensor \C) \mid (x,x)=0, (x,\bar x)>0 \}.
\end{equation}
It parametrizes Hodge structures of weight $2$ of primitive symplectic varieties with $(H^2(X,\Z),q_X) \isom \Lambda$. For $p\in \Omega_\Lambda$, we will denote the corresponding Hodge decomposition by
\[
\Lambda\tensor \C = H^{2,0}_p \oplus H^{1,1}_p \oplus H^{0,2}_p.
\]

\begin{definition}\label{definition rational rank}
The \emph{rational rank} of a period $p\in\Omega_\Lambda$ is defined as 
$$
\rrk(p):=\dim_\Q \left(\left(H_p^{2,0}\oplus H_p^{0,2}\right) \cap \Lambda\tensor\Q\right) \in \{0,1,2\}.
$$
We define the rational rank of a primitive symplectic variety $X$, denoted by $\rrk(X)$, to be the rational rank of its period $\mu_\C(H^{2,0}(X))$ after having chosen some \emph{marking}, that is, an isometry $\mu\colon H^2(X,\Z) \to \Lambda$. Note that the rational rank of $X$ does not depend on the choice of marking.
\end{definition}

In \cite[Theorem 4.8]{Ver15} and \cite[Theorem 2.5]{Ver17}, Verbitsky classified the possible orbits under the action of any arithmetic group; see also \cite[Proposition~3.11]{BL21}.

\begin{theorem}[Verbitsky]\label{theorem verbitsky}
Assume $\rk(\Lambda)\geq 5$. For $p\in\Omega_\Lambda$, there are three types of orbits of $p$ under the action of $\Gamma := \O(\Lambda)$, depending on the rational rank:
\begin{enumerate}
\item If\, $\rrk(p)=0$, then the orbit is dense, i.e., $\overline{\Gamma\cdot p}=\Omega_\Lambda$. 
\item If\, $\rrk(p)=1$, then $\overline{\Gamma\cdot p}$ is a $($countable$)$ union of totally real submanifolds of\, $\Omega_\Lambda$ of real dimension equal to $\dim_\C\Omega_\Lambda$. 
\item If\, $\rrk(p)=2$, then the orbit is closed, i.e.,  $\overline{\Gamma\cdot p}$ is countable.
\end{enumerate} 
\end{theorem}

Clearly, a general period has rational rank $0$. Periods of rational rank $2$ are said to have \emph{maximal Picard rank}.

\begin{theorem} \label{theorem main}
Let $X$ be a primitive symplectic variety. Suppose that every primitive symplectic variety which is a locally trivial deformation of\, $X$ satisfies the rational SYZ conjecture. Then the following hold:  
\begin{enumerate}
    \item\label{tmain-1} If\, $b_2(X) \geq 5$, then $X$ is non-hyperbolic.
    \item\label{tmain-2} If\, $b_2(X) \geq 5+\rrk(X)$, then $d_X \equiv 0$.
\end{enumerate}
\end{theorem}

The proof of this theorem will occupy the rest of the section. The main geometric ingredient for the proof of Theorem~\ref{theorem main} is the following result.

\begin{theorem} \label{theorem lagrangian fibration}
Let $X$ be a projective primitive symplectic variety with $b_2(X)\geq 5$, and let $L\in \Movbar(X)$ be a non-trivial line bundle on $X$ with $q_X(L)=0$. If\, $X$ satisfies the rational SYZ conjecture, then $d_{Z}\equiv 0$ for every compact variety $Z$ birational to $X$.
\end{theorem} 
\begin{proof}
We denote by $f\colon X\ratl B$ the rational Lagrangian fibration induced by $L$. The class $[L]\in \Pic(X)$ is isotropic, so taking a rational plane containing it and passing through the interior of the positive cone, one finds a non-proportional isotropic class $[L']\in \Pic(X)$ whose sign we choose in such a way that $[L']$ lies on the boundary of the positive cone. Note that $X$ has Picard rank at least $2$. 

We make a case distinction, depending on whether $L'\in \Movbar(X)$. If this is the case, $L'$ gives rise to a rational Lagrangian fibration \mbox{$f'\colon X \ratl B'$}. As the classes of $L$ and~$L'$ are not proportional, the maps $f$ and~$f'$ are distinct rational maps. Hence, $X$ is covered by families $\gothF, \gothF'$ of cycles such that the generic member of either family is birational to an abelian variety by Theorem~\ref{theorem matsushita} and Lemma~\ref{lemma lmp}. Moreover, these families are distinct. So if we consider the family $\gothF \cup \gothF'$ of analytic cycles, we obtain an almost holomorphic map $g\colon X \ratl S$ by Campana's Theorem~\ref{theorem campana}. Clearly, the fibers of $g$ have dimension strictly greater than $\dim X/2$. Hence, by Theorem~\ref{theorem almost holomorphic matsushita}, the base $S$ must be a point. In particular, $X$ is chain connected by cycles in $\gothF \cup \gothF'$. Since $d_X\vert_F\equiv 0$
for every $F \in \gothF \cup \gothF'$, the claim follows. 

It remains to treat the case where $L'\notin \Movbar(X)$. We will argue by induction on the Picard rank of~$X$. For the inductive argument, we need $X$ to be $\Q$-factorial, which we may assume by replacing it by a $\Q$-factorialization; see \cite[Corollary~~1.4.3]{BCHM10}. This may increase the Picard number once, but during the inductive process, we will always remain $\Q$-factorial. 

First of all, let us observe that in case $f$ is not almost holomorphic, a similar argument as above shows that $d_X\equiv 0$. Indeed, Campana's theorem applied to the family~$\gothF$ of analytic cycles given by the fibers of $f$ will result in an almost holomorphic map with fiber dimension strictly greater than $\dim X/2$. Thus, from now on, whenever a Lagrangian fibration appears, we may assume it to be almost holomorphic. 

As $L'\notin\Movbar(X)$, by \cite[Proposition~5.8]{LMP22a}, there exists a prime exceptional divisor $E$ on $X$ and $q_X(E,L)\geq~0$ as $L\in \Movbar(X)$.  By the argument of \cite[Th\'eor\`eme~3.3]{Dru11}, the divisor $E$ can be contracted after a sequence of flips. In particular, there exist a birational map $\phi\colon X \ratl X'$ and a contraction $\pi\colon X' \to \bar X$ of~$E$ such that $\phi$ is an isomorphism in codimension~$1$ and $\pi$ is an isomorphism outside $E$. In particular, the Picard rank of $X'$ coincides with that of $X$. We conclude that $E$ is uniruled. Both $X'$ and $\bar X$ are $\Q$-factorial primitive symplectic varieties. 

Let us first assume that $q_X(E,L)>0$. From \cite[Remark~3.11]{LMP22a}, we infer that~$L$ has positive degree on the general ruling curve $R$ of $E$, so $R$ cannot be contracted by~$f$. As above, we apply Campana's theorem to the family of fibers of $f$ together with the ruling curves of $E$. 

We are left with the case $q_X(E,L)=0$. Then $E$ is vertical with respect to the almost holomorphic Lagrangian fibration $f$. We replace $X$ by the target of the contraction $\pi\colon X \to \bar X$ of $E$ and $L$ by the $\Q$-line bundle $\bar L:=\pi_*\phi_* L$. If $\bar f:= \pi \circ \phi \circ f$ is almost holomorphic, then $\bar L$ still has BBF square zero (it cannot be positive as it intersects a general curve in the fibers of $\bar f$ trivially). Now, the Picard rank of $\bar X$ is strictly smaller than that of $X$, and we proceed inductively. Note that if $\rk\Pic(X)=2$, the $\Q$-line bundle $\bar L$ becomes ample on $\bar X$, so the fibration $\bar f\colon \bar X \ratl B$ cannot be almost holomorphic. 

We will explain next how to deduce $d_X \equiv 0$. Note that for singular varieties, the vanishing of the Kobayashi pseudometric is not a birational invariant in general. The inductive argument above produces a diagram
\[
\xymatrix{
& Y \ar[dl]_{\pi}\ar[dr]^{\pi'}&\\
X\ar@{-->}[rr]^\psi&&X''\rlap{,}\\
}
\]
where $X''$ is primitive symplectic, $\psi$ is birational, $Y$ is a resolution of indeterminacy of~$\psi$, and $X''$ is chain connected by a family $\gothG$ of analytic cycles satisfying the following two properties. For every $G\in \gothG$, we have $d_{X''}\vert_G \equiv 0$, and for every irreducible component $\gothG_0 \subset \gothG$, the locus covered by cycles in $\gothG_0$ has codimension at most~$1$ (as they come either  from fibers of rational fibrations or from rulings of uniruled divisors). Hence, the family $\gothG$ lifts to $Y$. There we add the $\pi'$-exceptional rational curves and obtain a family of cycles for which $Y$ is chain connected. Note that fibers of $\pi'$ are rationally chain connected by \cite[Corollary~1.5]{HM07}. We deduce $d_Y \equiv 0$ and thus also $d_X\equiv 0$. As we could have replaced $Y$ by a further blowup, the claim about the vanishing of the Kobayashi metric on a birational model of $X$ follows.
\end{proof}

Now that we have established the vanishing of the Kobayashi pseudometric for primitive symplectic varieties admitting Lagrangian fibrations, we use an ergodicity argument to transport this property to all varieties in the same component of the moduli space. For this, we need the following preliminary consideration. 

Let $X$ be a primitive symplectic variety, and let $\mu$ be a marking on $X$. Associated to the pair $(X,\mu)$, there is the \emph{monodromy group} $\Mon \subset \O(\Lambda)$. It is defined to be the image under the morphism $\O(H^2(X,\Z))\to \O(\Lambda)$ induced by the marking $\mu$ of the group of automorphisms of $H^2(X,\Z)$ that arise by parallel transport in locally trivial families. This group only depends on the connected component of the marked moduli space containing $(X,\mu)$; we refer to \cite[Section~8]{BL22} for more details.

Let $X'$ be a primitive symplectic variety which is equivalent to $X$ by locally trivial deformations. We endow $X'$ with a marking $\mu'$ that is obtained from the one of $X$ by parallel transport. Let us denote by $p:=\mu(H^{2,0}(X))$, $p':=\mu'(H^{2,0}(X')) \in\Omega_\Lambda$ the periods of $X, X'$ thus obtained. Here, $\Lambda$ is a lattice isometric to $(H^2(X,\Z),q_X)$. 

\begin{definition}\label{definition monodromy orbit closure}
We say that $X$ \emph{is in the $\Mon$-orbit closure of} $X'$ if $p\in \ol{\Mon\cdot p'}$. 
\end{definition}

Note that this definition does not depend on the choice of $\mu$ as long as $\mu'$ is chosen as explained above. The following is the analog of \cite[Theorem~2.1]{KLV14} in the smooth case. The idea of the proof is essentially the same, but for convenience, we spell out the details.

\begin{proposition} \label{proposition usc} 
Let $X$ be a projective primitive symplectic variety with a rational Lagrangian fibration induced by a line bundle. Assume that $b_2(X) \geq 5$ and that the rational SYZ conjecture holds. Then every primitive symplectic variety $X'$ locally trivially deformation equivalent to $X$ such that $X$ is in the $\Mon$-orbit closure of\, $X'$  satisfies $d_{X'}\equiv 0$.
\end{proposition} 
\begin{proof}
Let $\scrX \to \Def^\lt(X)=:S$ be the universal deformation of $X$, and let $\pi\colon  \scrY \to \scrX$ be a simultaneous resolution of singularities, which exists by \cite[Corollary~2.27]{BGL22}. Let us denote by $\pi_0\colon Y:=\scrY_0\to X$ the central fiber and consider the diameter function 
$$\diam\colon  S \longrightarrow {\mathbb R}_{\ge 0}, \quad s\longmapsto \diam(\scrY_s)$$
for the Kobayashi pseudometric. It was shown in \cite[Theorem~2.1]{KLV14} that $\diam$ is upper semi-continuous for families of smooth varieties. Hence, for all $\eps >0$, the sets 
$$U_\veps:=\{s\in S \mid \diam(\scrY_s) < \veps\}$$
are open (and non-empty, as $0\in U_\veps$ by Theorem~\ref{theorem lagrangian fibration}). By the local Torelli theorem, see \cite[Proposition~5.5]{BL22}, we can identify $S$ with a small open set in the period domain $\Omega_\Lambda$, where $\Lambda \isom (H^2(X,\Z),q_X)$.  Let us consider the action of the monodromy group $\Mon \subset  \Gamma$  of $X$ on $\Omega_\Lambda$. Let $X'$ be as in the statement of the proposition, and let us adopt the notation of Definition~\ref{definition monodromy orbit closure}. Then the $\Mon$-orbit of the period $p'\in \Omega_\Lambda$ of $X'$ has the period $p\in \Omega_\Lambda\cap S$ of $X$ in its closure. 

The sets $U_\veps$ are saturated for the $\Mon$-action in the sense that $(\Mon.U_\veps) \cap S = U_\veps$. In particular, the set $\Mon\cdot p' \cap S$ is contained in $U_\veps$ for every $\veps >0$. It follows that $\diam(\scrY_s) =0$ for all $s\in \Mon\cdot p' \cap S$. For each such $s$,  the global Torelli theorem, see \cite[Theorem~1.1]{BL22}, implies that $X'$ and $\scrY_s$ are bimeromorphic. Let us choose a bimeromorphism $\scrY_s \ratl X'$ and a resolution of indeterminacies, \textit{i.e.}, a diagram 
\[
\xymatrix{
&W\ar@/^/[dr]^p\ar@/_/[dl]_q&\\
 \scrY_s \ar@{-->}[rr] &  & X\rlap{,}\\
}
\]
where $p,q$ are bimeromorphic morphisms and $W$ is a smooth and compact variety. By item \eqref{item bimeromorphic to smooth} of Lemma~\ref{lemma properties hyperbolicity}, the variety $W$ also has Kobayashi diameter $0$. But $p$ is distance decreasing and surjective, so the same holds for $X$ and the claim follows.
\end{proof}

\begin{proof}[Proof of Theorem~\ref{theorem main}]
We start by proving~\eqref{tm-2}. In view of Proposition~\ref{proposition usc}, we need to find a small locally trivial deformation $Y$ of $X$ which is projective, admits a rational Lagrangian fibration, and is contained in the $\Mon$-orbit closure of $X$. Note that the subgroup $\Mon\subset \Gamma$ has finite index by \cite[Theorem~1.1]{BL22}, so the analogous trichotomy to Theorem~\ref{theorem verbitsky} holds for orbit closures of $\Mon$. 

Let us choose a marking $\mu$ on $X$, and let us fix a lattice $\Lambda$ that is isometric to $(H^2(X,\Z),q_X)$. The assumption $b_2(X)\geq 5+\rrk(X)$, together with Meyer's theorem, shows that there is an isotropic class 
\[
\alpha\in \Lambda \cap \mu\left(H^2(X,\Q) \cap \left(H^{2,0}(X) \oplus H^{0,2}(X)\right)\right)^\perp.
\]
If $\rrk(X)=2$, the class $\alpha$ is of type $(1,1)$ on $X$ itself, and the rational SYZ conjecture allows us to conclude. If $\rrk(X)=0$, we first choose a period in $\alpha^\perp$ of rational rank at most $1$ and hence find a primitive symplectic variety $Y$ in the same component of the marked moduli space realizing that period, by~\cite[Theorem~1.1]{BL22}. We may assume $Y$ to be projective by \cite[Corollary~6.10]{BL22} and the assumption on $b_2(X)$. It remains to treat the case where $\rrk(X)=1$. Let $\lambda$ be a generator of  $H^2(X,\Q) \cap (H^{2,0}(X) \oplus H^{0,2}(X))$. As $\lambda \in \alpha^\perp$, we just need to choose any other $\mu \in \Lambda_\R \cap \alpha^\perp$ for which $\left\langle\lambda, \mu\right\rangle$ is a positive $2$-space and not rational. This $2$-space defines a period of rational rank $1$ that is in the orbit closure of $X$. The sought-for variety $Y$ is again obtained from \cite[Theorem~1.1]{BL22}, and hence~\eqref{tm-2} follows.

Finally,~\eqref{tmain-1} follows from~\eqref{tmain-2}, as the property of being non-hyperbolic is closed in families; see Theorem~\ref{Brody theorem}.
\end{proof}

While primitive symplectic varieties form a large class of singular holomorphic symplectic varieties, one may wonder whether assuming primitivity is really necessary. The following observation shows that it is indeed superfluous. By a \emph{holomorphic symplectic variety}, we mean a variety with rational singularities having a holomorphic symplectic form on the regular part.

\begin{proposition}\label{proposition decomposition theorem}
If the Kobayashi pseudometric vanishes for every irreducible symplectic variety, then the same holds true for any compact K\"ahler holomorphic symplectic variety.
\end{proposition}
\begin{proof}
Let $X$ be a compact K\"ahler holomorphic symplectic variety. By the decomposition theorem \cite[Theorem~A]{BGL22}, we know that, up to a finite quasi-\'etale cover, $X$ is a product of irreducible symplectic varieties and complex tori of even dimension. The finite cover is distance decreasing, so we are reduced to showing the claim separately for tori and irreducible symplectic varieties; see Lemma~\ref{lemma properties hyperbolicity}. For tori, the claim is obvious, and for irreducible symplectic factors, the claim holds by assumption.
\end{proof}

\section{Applications and examples}\label{section examples}

Here, we discuss some examples of (orbifold) primitive symplectic varieties with small second Betti numbers $b_2$. Still, it is possible to show the vanishing of their Kobayashi pseudometrics as they are quotients of primitive symplectic varieties with $b_2 \geq 7$ so that our result applies to the covering variety (if the covering variety has $b_2\geq 13$, one can of course also use \cite{KLV14}). We also discuss some crepant partial resolutions of these quotients.

\begin{example}
Fu and Menet \cite[Example 5.2]{FM21} construct the following quotients based on Mongardi's Ph.D.\ thesis work \cite[Section 4.5]{Mon13}: $M_{11}^i = X_i/\sigma_i$, where $X_i$ is a $K3^{[2]}$-type manifold endowed with a special symplectic automorphism $\sigma_i$ of order $11$, for $i =1,2$. Both primitive symplectic orbifolds $M_{11}^i$ have second Betti number $b_2(M_{11}^i)=3$. Since the Kobayashi pseudometric of $X_i$ vanishes by \cite[Remark~1.2, Theorem~1.3]{KLV14}, the Kobayashi pseudometric of the quotients $M_{11}^i$ also vanishes by Lemma~\ref{lemma properties hyperbolicity}.  

Similarly, based on Mongardi's Ph.D.\ thesis work \cite[Section 4.4]{Mon13}, Fu and Menet \cite[Example~5.3]{FM21} construct a quotient $M_7 = X / \sigma$, where $X$ is a  $K3^{[2]}$-type manifold endowed with a symplectic automorphism~$\sigma$ of order $7$. In this case, $b_2(M_{7})=5$. As above, one concludes that the Kobayashi pseudometric of~$M_7$ vanishes.
\end{example}

\begin{example} \label{Giovanni example}
We learned the following example from Giovanni Mongardi. Let $S$ be Fermat's quartic $K3$ surface, and let us consider the symmetries coming from the symmetries of its defining equation. The automorphism group of $S$ as a projective variety in $\mathbb P^3$ can be computed as $\Aut (S) = (\mathbb Z / 4 \mathbb Z)^3 \rtimes S_4$, where $S_4$ is the symmetric group. Not all of these automorphisms preserve the symplectic form of $S$. The group of symplectic automorphisms of $S$ is the kernel of the natural homomorphism $\Aut (S) = (\mathbb Z / 4 \mathbb Z)^3 \rtimes S_4 \rightarrow \mathbb C^*$, which is denoted by $F_{384}$, and it is a subgroup of order $384$ of the Mathieu group $M_{24}$; see \cite{Muk88}. 
Let $n \geq 2$, let $G$ be the induced group of symplectic automorphisms of $S^{[n]}$ preserving the degree~$4$ polarization, and let $X := S^{[n]}/G$. Then, $b_2(X)=4$ by \cite{Has12}; see in particular Section~10.3 there ($F_{384}$ is the group number $80$ in the list). Since the Kobayashi pseudometric of $S^{[n]}$ vanishes by \cite[Remark~1.2, Theorem~1.3]{KLV14}, the Kobayashi pseudometric of the quotient $X$ also vanishes by Lemma~\ref{lemma properties hyperbolicity}. 
\end{example}

\begin{example}
Let $T$ be a complex 2-torus equipped with a symplectic automorphism~$\sigma_4$ of order $4$ as constructed by Fu and Menet \cite[Example 5.4]{FM21}. Let $K_2(T)$ be the generalized Kummer variety associated to $T$, and let $\sigma_4^{[2]}$ be the automorphism extending $\sigma_4$ on $K_2(T)$. Fu and Menet construct a proper birational map  \mbox{$K_4'\to K_2(T) / \sigma_4^{[2]}$}, where $K_4'$ is a crepant resolution in codimension $2$. The primitive symplectic orbifold~$K_4'$ has second Betti number $b_2(K_4')=6$ and by construction is dominated by a blowup of $K_2(T)$ so that its Kobayashi pseudometric vanishes by Lemma~\ref{lemma properties hyperbolicity}. We use that the Kobayashi pseudometric of $K_2(T)$ vanishes by Theorem~\ref{theorem main}. 

The question remains whether the Kobayashi pseudometric also vanishes on all locally trivial deformations of $K_4'$ as the generic such deformation will no longer be birational to a quotient of a generalized Kummer variety. We do not know whether the SYZ conjecture holds for deformations of $K_4'$. However, instead of Lagrangian fibered varieties, we can use quotients as an input and then argue as in Proposition~\ref{proposition usc}. With this modification, the argument of our main result Theorem~\ref{theorem main} implies that all deformations of $K_4'$ are non-hyperbolic and that all of them except for maybe those with maximal Picard rank have vanishing Kobayashi pseudometric.
\end{example}

\begin{example}
Let $X$ be a projective fourfold of $K3^{[2]}$-type admitting a symplectic involution $\iota$. The moduli space of such pairs of objects $(X, \iota)$ is described in \cite[Sections~2 and 3]{CGKK21}. The fixed loci of symplectic involutions of $K3^{[2]}$-type manifolds are classified in \cite[Theorem 4.1]{Mongardi_invol}, and, more generally, the fixed loci of symplectic involutions of $K3^{[n]}$-type manifolds are classified in \cite[Theorem 1.1]{KMO22}. The irreducible symplectic orbifolds $Y \rightarrow X/ \iota$, obtained as a partial resolution of $X/ \iota$ for certain fourfolds $X$ of $K3^{[2]}$-type and for certain symplectic involutions $\iota$, are called {\it Nikulin orbifolds}; see \textit{e.g.}~\cite[Definition 3.1]{CGKK21}. Menet \cite[Theorem 2.5]{Men15} has computed their integral second cohomology, and $b_2(Y)=16$. 
Since the Kobayashi pseudometric of $X$ vanishes by \cite[Remark 1.2, Theorem 1.3]{KLV14}, the Kobayashi pseudometric of the quotients $X/ \iota$ and the partial resolutions $Y$ also vanishes by Lemma~\ref{lemma properties hyperbolicity}.  

Note that a general deformation of $Y$ is no longer a partial resolution of the quotient~$X/\iota$, so \cite{KLV14} no longer applies. Our main result Theorem~\ref{theorem main} would guarantee the vanishing of the Kobayashi pseudometric if the rational SYZ conjecture were satisfied. Verifying the (rational) SYZ conjecture for this class of examples seems to be a interesting and valuable task.
\end{example}


\newcommand{\etalchar}[1]{$^{#1}$}
\providecommand{\bysame}{\leavevmode\hbox to3em{\hrulefill}\thinspace}

\end{document}